\documentclass[11pt,reqno]{amsart}

\newcommand{\K}{{\mathbf K}^\crit}

\newcommand{\szego}{Szeg\"o }

\newcommand{\kahler}{K\"ahler }

\newcommand{\PP}{{\mathbb P}}
\newcommand{\R}{{\mathbb R}}
\newcommand{\C}{{\mathbb C}}

\newcommand{\Z}{{\mathbb Z}}

\newcommand{\dbar}{\bar\partial}
\newcommand{\ddbar}{\partial\dbar}

\newcommand{\diag}{{\operatorname{diag}}}

\renewcommand{\H}{{\mathbf H}}

\newcommand{\FS}{{{\operatorname{FS}}}}

\renewcommand{\phi}{\varphi}

\newcommand{\bcal}{\mathcal{B}}

\newcommand{\fcal}{\mathcal{F}}
\newcommand{\gcal}{\mathcal{G}}
\newcommand{\hcal}{\mathcal{H}}
\newcommand{\ical}{\mathcal{I}}

\newcommand{\lcal}{\mathcal{L}}

\newcommand{\pcal}{\mathcal{P}}
\newcommand{\rcal}{\mathcal{R}}
\newcommand{\ocal}{\mathcal{O}}
\newcommand{\qcal}{\mathcal{Q}}

\newcommand{\al}{\alpha}
\newcommand{\be}{\beta}

\def    \Z  {{\mathbb Z}}
\def    \R  {{\mathbb R}}
\def    \C  {{\mathbb C}}

\newtheorem{theo}{{\sc Theorem}}[section]
\newtheorem{cor}[theo]{{\sc Corollary}}

\newtheorem{lem}[theo]{{\sc Lemma}}
\newtheorem{prop}[theo]{{\sc Proposition}}

\newenvironment{defin}{\medskip\noindent{\it Definition:\/} }{\medskip}

\def\HungarianAccent#1{{\accent"7D #1}}
\def\o{\omega}

\def\Szego{Szeg\HungarianAccent{o} }
\def\K{K\"ahler }

\def\ra{\rightarrow}
\def\a{\alpha}
\def\th{\theta}
\def\vp{\varphi}
\def\pa{\partial}
\def\isom{\cong}

\def\w{\wedge}
\def\i{\sqrt{-1}}
\def\text{\textstyle}
\def\equationspace{\medskip\noindent}

\def\ra{\rightarrow}

\def\isom{\cong}

\def\dis{\displaystyle}

\def\calQ{\qcal}
\def\calP{\pcal}
\def\calR{\rcal}
\def\calL{\lcal}
 \def\H{\hcal}

\newcommand{\Hilb}{{{\operatorname{Hilb}}}}
\newcommand{\BF}{{{\operatorname{BF}}}}
\title[Bergman approximations of harmonic maps 
]
{Bergman approximations of harmonic maps into the space of K\"ahler metrics on toric varieties}

\author{Yanir A. Rubinstein }
\author{Steve Zelditch }
\address{Department of Mathematics, Massachusetts Institute of Technology,
Cambridge, MA 02139, USA}
\curraddr{Department of Mathematics, Princeton University, Princeton, NJ 08544, USA}
 \email{yanir@member.ams.org}
\address{Department of Mathematics, Johns Hopkins University, Baltimore,
MD 21218, USA} \email{zelditch@math.jhu.edu}

\thanks{\hglue-12pt March 8, 2008}
\thanks{\hglue-12pt Mathematics Subject Classification (2000): 
Primary 32Q15. 
Secondary 14M25, 
32W20, 
53D50, 
58B20, 
58E20.% 
}

\begin{document}

\maketitle

\begin{abstract}
We generalize the results of Song-Zelditch on geodesics in spaces
of K\"ahler metrics on toric varieties to harmonic maps of any
compact Riemannian manifold with boundary into the space of
K\"ahler metrics on a toric variety. 
We show that the harmonic map equation can always be solved and that 
such maps may be  approximated in the $C^2$ topology by harmonic
maps into the spaces of Bergman metrics. In particular,
WZW maps, or equivalently solutions of a homogeneous
Monge-Amp\`ere equation on the product of the manifold with a
Riemann surface with $S^1$ boundary admit such approximations.
We also show that the Eells-Sampson flow on the space of \K potentials
is transformed to the usual heat flow on the space of symplectic potentials under
the Legendre transform.
\end{abstract}

\section{Introduction}

Our main purpose in this article is to prove that the Dirichlet problem
for a harmonic map $\phi: N\ra \hcal(T)$ of any compact Riemannian manifold with boundary
$N$ into the infinite-dimensional space $\hcal(T)$ of toric  K\"ahler
metrics on a smooth projective toric variety $(M, \omega)$ admits a smooth
solution that may be
approximated in $C^2(N \times M)$ by a special squence of harmonic
maps $\phi_k: N \to \bcal_k(T)\subset \H(T)$ into the finite-dimensional subspaces
of Bergman (or Fubini-Study)  metrics induced from projective
embeddings. As a special case, we show that the WZW  (Wess-Zumino-Witten) equation,
 or equivalently the homogeneous complex
Monge-Amp\`ere equation on the product of the manifold with a
Riemann surface with $S^1$ boundary, admits such a solution as well as such approximations. 
This generalizes previous work of Song-Zelditch in the
case of geodesics, i.e., where $N = [0, 1]$.

Before stating our results, we briefly recall the background to
our problem. Let $(M,\o)$ be a compact closed \K manifold of
dimension $m$  with integral \kahler form and let $(L, h_0)  \to
M$ be an ample Hermitian  holomorphic line bundle with
$\omega_{h_0} = \omega$ satisfying $[\o]= c_1(L)$, where
$\omega_{h_0}=-\frac{\i}{2\pi}\ddbar\log h_0$ is the curvature $(1,1)$-form of $h_0$. Any other
hermitian metric on $L$ may be expressed as  $h_{\phi} = e^{-\phi}
h_0$, with $\vp$ a smooth function on $M$.   Following Mabuchi, Semmes and Donaldson \cite{M,S2,D1} one
may regard the space

\begin{equation}
 \label{HCALDEF} \hcal_{\o}:= \{\phi\in C^{\infty}
(M) : \omega_\phi\ = \ \omega+\frac{\sqrt{-1}}{2\pi}\ddbar \phi>0\
\}
\end{equation}

\noindent
of potentials of \K metrics in a fixed cohomology class as an
infinite-dimensional symmetric space dual to the group of Hamiltonian
diffeomorphisms $\hbox{Ham}(M,\o)$. We will henceforth usually identify $\hcal_{\o}$
with the space of Hermitian metrics on $L$ of positive curvature  

$$
\hcal:=\{\,h\,:\, h=h_0e^{-\vp}, \vp\in\hcal_{\o}\}.
$$

\equationspace
The symmetric space Riemannian metric $g_{L^2}$
is defined by

\begin{equation} \label{metric} 
g_{L^2}(\zeta,\eta)_{\vp}:= \frac1V\int_M
\zeta\eta\, {\omega_{\phi}^m},\qquad \phi \in \hcal_{\o},\quad
\zeta,\eta \in T_{\phi} \hcal_{\o}\isom C^\infty(M).
\end{equation}

\equationspace
A basic idea, considered by Yau, Tian and Donaldson, is to approximate
transcendental objects defined on $\hcal$ by algebraic
objects defined on the finite dimensional symmetric spaces
$\bcal_k$ of Bergman (or Fubini-Study) metrics on $L$. To define
them, we use the following notation:
$H^0(M, L^k)$ is the space of  holomorphic
sections of the $k$-th power $L^k \to M$,  $d_k+1= \dim H^0(M,
L^k)$ and  $\bcal H^0(M, L^k)$ is  the manifold of all bases
$\underline{s} = \{s_0, \dots, s_{d_k}\}$ of $H^0(M, L^k)$. A
basis $\underline{s}$ determines a Bergman metric
\begin{equation} \label{FSDEF}  
h_{\underline{s}} := (\iota_{\underline{s}}^\star
h_{\FS})^{1/k} = \frac{h_0}{\left( \sum_{j = 0}^{d_k}
|s_j(z)|^2_{h_0^k} \right)^{1/k}}, \end{equation} 
as the pullback
of the Fubini-Study metric $h_{\FS}$ on the hyperplane bundle $\ocal(1) \to \PP^{d_k}$
under the Kodaira embedding
\begin{equation} \label{KODEMB} \iota_{\underline{s}}: M \to \PP^{d_k},\;\;z \to [s_0(z),
\dots, s_{d_k}(z)].  \end{equation}  
The space of all  Bergman metrics defined by  $\bcal H^0(M, L^k)$ is denoted by
\begin{equation}
 \label{BERGMETDEF} \; {\mathcal
B}_k = \{h_{\underline{s}}, \;\; \underline{s} \in \bcal H^0(M,
L^k) \}.
\end{equation}
We may  identify the space ${\mathcal B}_k$ with the symmetric space $GL(d_k + 1, \C) / U(d_k + 1)$ 
since $GL(d_k + 1, \C)$ acts
transitively on the set of bases, while $\iota_{\underline{s}}^\star
h_{\FS}$ is unchanged if we replace the basis $\underline{s}$ by a
unitary change of basis.

A natural question is:  to what
extent can the geometry of $\hcal$ be approximated by that of the
spaces $\bcal_k$ of `algebro-geometric' metrics? At the most basic
level of individual points, it  follows by the Tian Asymptotic
Isometry Theorem \cite{T}, and its subsequent refinements \cite{C,Z},
that a metric $h \in \hcal$ can be approximated in a canonical way
by a sequence of Bergman metrics $h_k$. 
The geometry of a Riemannian manifold is reflected to a large extent
by its geodesics and more generally by the specification of the harmonic maps
into it, involving the analysis of certain nonlinear elliptic PDEs.
In this article we describe how solutions to these PDEs on $\hcal$ can be approximated in
an algebro-geometric manner by a sequence of solutions to PDEs on $\bcal_{k}$, in the setting
of a toric variety. 

We will need the canonical
sequence mentioned above to state our results, so we recall how it is constructed.
First, observe that  ${\mathcal B}_k$ is isomorphic to the
symmetric space  $\ical_k$ of
Hermitian inner products 
on $H^0(M, L^k)$, the correspondence being that a basis is
identified with an inner product for which the basis is 
Hermitian orthonormal. Define the maps
\begin{equation}
\label{HilbEq} \Hilb_k: \hcal \to \ical_k 
\end{equation}
by the rule that a Hermitian metric  $h \in \hcal$ induces the
metrics $h^k$ on $L^k$ and the   inner products on $H^0(M, L^k)$,
\begin{equation} \label{HILBDEF} 
||s||^2_{\Hilb_k(h)}:=  \frac1V\int_M |s(z)|_{h^k}^2 (k\o_h)^m,
\end{equation}
where $V=\int_M\o_h^m$.
An inner product $I =(\,\cdot\,,\,\cdot\,)$ on $H^0(M, L^k)$  determines an
$I$-orthonormal  basis $\underline{s} = \underline{s}_I$ of
$H^0(M, L^k)$, an associated Kodaira embedding (\ref{KODEMB}),
as well as a Bergman metric given by
\begin{equation} 
\label{FSDEF} 
\FS_k(I):= h_{\underline{s}_I}. 
\end{equation}
Tian's asymptotic isometry theorem then states that $\FS_k \circ \Hilb_k(h) \to h$ in
the $C^{\infty}(M)$ topology with complete asymptotic expansions \cite{T,C,Z}.

The question was then raised \cite{AT, D2, PS1} whether geodesics of
$\hcal$ could be well-approximated by the one-parameter
subgroup geodesics of $\bcal_k$. Geodesics of $\hcal$ are
given by solutions of

\begin{equation}
\label{GeodesicEq}
\ddot\vp-{\text\frac12}|\nabla\dot\vp|^2=0, 
\quad \vp(0,\,\cdot\,)=\vp_0,\; \vp(1,\,\cdot\,)=\vp_1,\quad \vp_0,\vp_1\in\hcal,
\end{equation}

\medskip
\noindent
where $\vp$ is considered as a map from $[0,1]$ to $\hcal$, or equivalently as a function
on $[0,1]\times M$.
Phong-Sturm 
studied this problem in depth and proved that a sequence of geodesics in $\bcal_k$ converges
weakly (almost everywhere) to a prescribed geodesic 
of $\hcal$ \cite{PS1}. 
Song-Zelditch proved that the same sequence converges in
$C^2([0, 1] \times M)$ when the manifold is toric and one restricts to the torus-invariant
metrics \cite{SoZ2}, and Berndtsson
used a different argument to prove that geodesics in $\hcal$
can be $C^0$-approximated by geodesics in spaces of Bergman
metrics induced by embeddings by sections of $L^k \otimes K_M$,
where $K_M$ is the canonical bundle of $M$ \cite{B}.
In addition, Phong-Sturm and Song-Zelditch
proved approximation results for geodesic rays constructed from
test configurations \cite{PS2,SoZ3}.

A harmonic map between two Riemannian
manifolds $(N,f)$ and $(\tilde N,\tilde f)$ is a critical point of the energy functional

$$
E(a)=\int_{N} |da|^2_{f\otimes a^\star \tilde f}\; dV_{N,f},
$$

\medskip

\noindent
on the space of smooth maps $a$ from $N$ to $\tilde N$ \cite{ES}.
The problem we study   in this article is whether   higher
dimensional harmonic maps of general compact Riemannian manifolds
$N$ with boundary $\partial N$ into $\hcal$ admit similar kinds of
`algebro-geometric' approximations.
For maps into the space of toric \K metrics on
toric manifolds, we obtain an affirmative solution at the
same level of precision as in the case of geodesics studied by Song-Zelditch.

To describe our results,  let us define our
setting more precisely.  We recall that a toric variety $M$ of
complex dimension $m$ carries the holomorphic action of a complex torus
$(\C^\star)^m$ with an open dense orbit. We let $T = (S^1)^m$ be the
associated real torus. Objects associated to $M$ are called toric
if they are invariant  with respect to $T$. We let $\omega$
denote a toric integral \kahler form on $M$, and  let $L$ be an
ample line bundle with $[\omega]=c_1(L)$. We then define
the space of toric Hermitian metrics on $L$,

\begin{equation}
\hcal(T) = \{h \in \hcal: t^* h = h, \;\; \forall t \in T\}.
\end{equation}

\equationspace
This is a flat submanifold of $\hcal$.
As before we will frequently identify an element of $\hcal(T)$ with
the corresponding element of $\hcal_\o$. Moreover, we will
often identify an element with the local \K potential defined on 
the open orbit of the complex torus; see Section \ref{LAG}.
We also denote by
$\bcal_{k}(T) \subset \bcal_k$ the subspace of Bergman metrics
defined by $T$ invariant inner products, or equivalently by
$T$-equivariant embeddings. Any such embedding is induced by
the basis of toric monomials $\{\chi_\a\}_{\a\in kP\cap\Z^m}$ of $H^0(M,L^k)$,
where $P$ denotes the moment polytope associated
to the action, and $kP$ denotes its dilation by a factor of $k$.
Finally, let $(N,f)$ be a compact oriented Riemannian manifold with
smooth boundary, let $G(y,q)$ denote the positive Dirichlet  Green kernel 
for the Laplacian $\Delta_N:=\Delta_{N,f}$ (see Section \ref{LAG}),
and let $dV_{\pa N,f}$ denote the induced measure on $\partial N$ from the restriction of the Riemannian
volume form $dV_{N,f}$ from $N$ to $\pa N$. The main result of this article is:

\begin{theo} 
Let $(M, L, \omega)$ be a polarized toric \kahler
manifold, and let $(N,f)$ be a compact oriented smooth
Riemannian manifold with smooth boundary $\partial N$. Let
$\psi:\partial N\rightarrow \hcal(T)$ denote a fixed smooth map.
There exists a harmonic map $\vp:N\rightarrow\hcal(T)$ with
$\vp|_{\partial N}=\psi$ and harmonic maps $\vp_k:N\rightarrow
\bcal_{k}(T)$ with $\vp_k|_{\partial N}=\FS_k\circ
\Hilb_k(\psi)$, given on the open orbit by
\begin{equation}
\label{ApproximateHarmonicMapsEq}
\vp_k(y,z)
=
\frac1k\log\sum_{\alpha\in kP\cap \Z^m}
|\chi_\alpha(z)|_{h_0^k}^2
\exp\Big(\int_{\partial N}\pa_{\nu_q}G(y,q)\log||\chi_\alpha||^2_{h^k_{\psi(q)}}dV_{\pa N,f}(q)\Big),
\end{equation}
and one has
$$
\lim_{k\rightarrow\infty}\vp_k=\vp,
$$
in the $C^2(N\times M)$ topology.
\end{theo}

A motivating special case is the unit disc $N=D:=\{z\in\C\,:\, |z|\le 1\}$. It has
been the subject of intensive studies (e.g., \cite{Ch,CT,D3}).
Then the map $\vp$ corresponds to certain foliations by
holomorphic discs arising from a solution of a certain homogeneous
complex Monge-Amp\`ere (HCMA) equation. To describe this case, let $\pi_2:D\times M\ra M$
denote the projection onto the second factor and consider the HCMA equation,

\begin{equation}
\label{HCMAFirstEq}
(\pi_2^\star\o+\i\ddbar\vp)^{n+1} =0, \quad \hbox{\ \ on \ \ } D\times M,
\end{equation}
\begin{equation}
\label{HCMASecondEq}
(\pi^\star\o+\i\ddbar\vp)|_{\{t\}\times M} >0,\quad \forall\, t\in D,
\end{equation}
\begin{equation}
\label{HCMAThirdEq}
\vp  = \psi,\quad \hbox{\ \ on\ \ } \pa D\times M.
\end{equation}

\equationspace
One may show that this HCMA is the Euler-Lagrange equation of
an infinite-dimensional version
of a Wess-Zumino-Witten model, given by the  energy functional

$$
E^{WZW}_{\sigma} (b) = \frac{1}{2} \int_D |\nabla b|^2 + \int_Z \theta,
$$

\equationspace
on the  space of maps $b \in C^\infty(D, G^{\C}/G)$, 
where $\sigma : \partial D \to
G^{\C}/ G$ is a fixed map $\sigma$ \cite{D1}. The Lie bracket of
$G$ determines a $3$-form $\theta$ and $Z$ is any cochain with
boundary $b(D) - b_0(D)$ for some fixed
reference map $b_0$ with the same boundary conditions $\psi$.
The Euler-Lagrange equations for this
functional are the WZW equations

\begin{equation} 
d^\star d\, b + [b_\star {\textstyle \frac{\pa}{\pa q}}, b_\star {\textstyle \frac{\pa}{\pa s}}] = 0, 
\end{equation}

\equationspace
in Euclidean coordinates $q + \i s  \in D$, and where $d^\star$ maps 
sections of $T^\star D\otimes b^\star TG^{\C}/G$ to sections of $b^\star TG^{\C}/G$.
Finally, when $G$ and $G^{\C}/G$ are replaced by $\hbox{Ham}(M,\o)$ and $\hcal$,
the Christoffel symbols are given by
$\Gamma(\zeta,\eta)|_\vp=-{\textstyle\frac12}g_\vp(\nabla \zeta,\nabla \eta)$, and
the WZW equation is

$$
\vp_{qq}+\vp_{ss}-{\textstyle\frac12}|\nabla\vp_q|^2-{\textstyle\frac12}|\nabla\vp_s|^2+\{\vp_q,\vp_s\}_{\o_\vp}=0.
$$

\equationspace
It is a perturbation of the usual harmonic map equation by a
Poisson bracket term. Coming back to the toric situation and
restricting to the space of torus-invariant \K potentials
$\hcal(T) \subseteq\hcal$ the functions $\vp_q$ and $\vp_s$ are
commuting Hamiltonians and hence the WZW equation reduces to the
harmonic map equation.
The finite-dimensional WZW equation on $GL(d_k+1,\C)/U(d_k+1)$ may be written similarly

$$
T^{-1}T_{qq}+T^{-1}T_{ss}-(T^{-1}T_q)^2-(T^{-1}T_s)^2+\i[T^{-1}T_q,T^{-1}T_s]=0.
$$

\equationspace
The torus-invariance then corresponds to restriction to diagonal
matrices and again the last term vanishes and the equation reduces
to the harmonic map equation. Geometrically, the curvature of
$\hcal$ comes from the Poisson bracket and when we restrict to the
flat subspace $\hcal(T)$ the noncommutativity disappears.

In the case of the unit disc, the normal derivative of the Green kernel
is the Poisson kernel, whose restriction to $D\times\partial D$ takes the form
$$
P(re^{\i\theta},e^{\i\gamma})=P_r(\theta-\gamma)=-{1\over{2\pi}}{{1-r^2}\over{1-2r\cos(\theta-\gamma)+r^2}}
$$
(our convention is that the Green function be nonnegative, 
as explained in Section \ref{LAG}). Then we have the following 
more explicit statement of Theorem 1.1:

\equationspace
\begin{cor} \label{LATTICE}
Let $(M, L, \omega)$ be a polarized toric \kahler
manifold.
Let $\vp$ be a solution of the HCMA equation (\ref{HCMAFirstEq})-(\ref{HCMAThirdEq})
with $\psi:S^1\ra \hcal(T)$ a smooth map, and let $\vp_k:N\ra \bcal_{k}(T)$ be given 
on the open orbit by

$$
\vp_k(re^{\i\gamma},z)
=
\frac1k\log\sum_{\alpha\in kP\cap \Z^m}
|\chi_\alpha(z)|_{h_0^k}^2
\exp\Big(\int_{\partial D}P_r(\theta-\gamma)\log||\chi_\alpha||^2_{h^k_{\psi(\theta)}}d\theta\Big).
$$

\equationspace
Then $\lim_{k\rightarrow\infty}\vp_k=\vp$ in the $C^2(D\times M)$ topology. 
\end{cor}

The proof of Theorem 1.1 builds upon the machinery developed by
Song-Zelditch for the study of geodesics in $\hcal(T)$. In the geodesic  case, i.e., $N = [0, 1]$,
the approximating 
Bergman \kahler potentials take the form

\begin{equation} 
\label{toricphik} 
\phi_k(t, z) = \frac{1}{k} \log
\sum_{\alpha \in k P \cap \Z^m} |\chi_{\alpha}(z)|^2_{h_0^k}
\;e^{-(1 - t) \log  ||\chi_{\alpha}||_{\Hilb_k(h_0)}^2 -t \log  ||\chi_{\alpha}||_{\Hilb_k(h_1)}^2}.
\end{equation}

\equationspace
We see that the straight line segment in the case $N =[0, 1]$ is replaced by the harmonic extension of the
boundary $L^2$ norming constants in the  general case. 
Aside from justifying the general formula,
we need to modify the estimates to apply to harmonic functions on $N$ rather than linear functions on
$[0, 1]$.  Using the  localization
lemma and the asymptotics of the peak values proved in \cite{SoZ2} (which we recall in 
Section \ref{BACKGROUND}),  the uniform convergence in $C^2$  reduces to a
verification of orders of amplitudes where the analysis is carried
out separately in the interior of the polytope and near its
boundary. Since many details are similar to the geodesic case, 
we concentrate here only on the novel features and the reader
of Section \ref{SectionProof} would benefit from some familiarity with \cite{SoZ2}.
The reader is also referred to \cite{R}, Chapter 3, where the results of the present article appear
with greater detail.

Let us make note of one more relation between the geodesic segment problem and the harmonic mapping problem.
In both cases a key aspect of the toric situation is that the Legendre transform linearizes the 
harmonic map equation. This was known previously for geodesics \cite{D1,G,S2} (see also \cite{SoZ2}
for a simple proof), but is observed for the first time
here for general harmonic maps. We refer the reader to Section \ref{LAG} 
where we also observe a generalization 
(\ref{LaplaceLegendreEq}) of a well known formula from convex analysis and 
show that the Eells-Sampson harmonic map flow is Legendre transformed to the usual heat flow. 
This fact is quite a remarkable property of the toric situation 
and does not hold for general variational problems.
It follows that one can explicitly solve the WZW and 
harmonic map equations in terms of the associated symplectic potentials. We make crucial use of this in proving 
the convergence of the Bergman harmonic maps, and it is the reason why we do not founder amid regularity problems
as in the general projective case. Our results give the first proof of convergence for higher dimensional
harmonic maps into spaces of \K metrics. 
We hope to discuss in a separate article convergence results for WZW maps on general projective
manifolds.

Y.A.R. would like to thank Gang Tian for suggesting this problem to him and for
his advice and warm encouragement, and Bo`az Klartag for enjoyable lectures and discussions 
on convex analysis. 
Both authors would like to thank the Technion, where this work
was begun, for its hospitality in Summer 2007.
This material is based upon work supported in part under a National Science 
Foundation Graduate Research Fellowship and grant DMS-0603850.

\section{\label{BACKGROUND} Background results}

We begin by  recalling  some basic facts regarding toric varieties
relevant to our setting. We refer the reader to \cite{Ca,STZ1,SoZ2}
for more details.

We will work with coordinates on the open dense orbit of the complex
torus given by $z=e^{\rho/2+\i\th}$, with $\rho,\th\in\R^m\times
(S^1)^m$. The real torus $T\isom (S^1)^m$ acts in a Hamiltonian
fashion with respect to $\o$. The image of its moment map $\mu$ is
a Delzant polytope $P\subset \R^m$. Let $x$ denote the Euclidean
coordinate on $P$. The polytope is given by $P=\{x\in\R^m\,:\,
l_r(x):=\langle x,v_r\rangle-\lambda_r\le0, r=1,\ldots,d\}$ where
$v_r$ is an outward pointing normal to the $r$-th
$(m-1)$-dimensional face of $P$ (also called a facet) and is a 
primitive element of the lattice $\Z^m$.

The toric monomials $\{\chi_\al(z):=z^\alpha\}_{\a\in kP\cap\Z^m}$
are an orthogonal basis of $H^0(M,L^k)$ with respect to any
element of $\bcal_{k}(T)$. Hence a toric inner product,
equivalently a point in $\bcal_{k}(T)$, is completely determined
by the $L^2$ norms (up to $k^n/V$), or  {\it norming constants},  of the toric
monomials---

$$
\qcal_{h^k}(\alpha):=||\chi_\alpha||^2_{h^k} = \int_{(\C^*)^m}
|z^{\alpha}|^2 e^{- k \phi} dV_h. 
$$ 

\equationspace
Unlike in the Introduction here we let $h=e^{-\vp}$ with $\vp$ a local \K potential
on the open orbit (that does not extend globally).
Define the normalized norms of the monomials

$$
\pcal_{h^k}(\alpha,z):=\frac{|\chi_\alpha(z)|^2_{h^k}}{||\chi_\alpha||^2_{h^k}},
$$

\equationspace
and their peak values

\begin{equation}
\label{PalphaEq}
\pcal_{h^k}(\alpha):=\frac{|\chi_\alpha(\mu_{h^k}^{-1}(\frac\alpha k))|^2_{h^k}}{||\chi_\alpha||^2_{h^k}}.
\end{equation}

\equationspace
In order to complete the proof of Theorem 1.1 we will need some of
the tools developed by Song-Zelditch, that we now recall. First, an
asymptotic expression for $\pcal_{h^k}(\a)$ for families of toric
Bergman metrics. This expression is sensitive to the distance of 
$\a/ k$ to the boundary of the polytope. Recall that the $k$-th Bargmann-Fock
model on $(\C, \i dz\w d\bar z)$ 
is given by the holomorphic functions that are $L^2$ with respect
to the Hermitian metric $h^k_{\BF}=e^{-k|z|^2}$ and a basis is given by all monomials
$z^\a$ with $\a\in\Z_+$. One may compute that
\begin{equation}
\label{BFEq}
\pcal_{h^k_{\BF}}(\a)=ke^{-\a}\frac{\a^\a}{\a!}.
\end{equation}

Let $\delta_k=\frac1{\sqrt{k}\log k}$. 
Denote by $\fcal_{\delta_k}(x)=\{r\,:\, l_r(x)<\delta_k\}$ the index
set for those facets to which $x$ is $\frac1{\sqrt{k}\log k}$-close,
and let $\delta^\sharp_k(x)$ denote the cardinality of this set. 
Set 
$$
\gcal_\vp(x):=\left(\delta_\vp(x)\cdot\Pi_{j\not\in\fcal_{\delta_k(x)}}l_j(x) \right)^{-1},
$$
where $\delta_\vp(x)$ is defined in (\ref{HessianSymplecticPotentialEq}) below and put
$$
\pcal_{\BF,\delta_k}(\a):=\Pi_{j\in\fcal_{\delta_k(x)}}\pcal_{h^k_{\BF}}(\a_j).
$$

\noindent
These two terms are the far and near contributions to the asymptotics of the peak values $\pcal_{h^k}(\a)$:

\begin{lem} {\rm(See \cite{SoZ2}, Propositions 6.1, 6.5.)}
\label{PLemma} Let $\delta_k=\frac1{\sqrt{k}\log k}$.
Let $\{h_t\}_{t\in K}$ be a family of metrics with
$K$ compact. Then there exist $C>0$ independent of $t$ such that
for any $\delta\in(0,\frac12)$

\begin{equation}\label{PEq}
\pcal_{h_t^k}(\a)=Ck^{\frac12(m-\delta^\sharp_k({\textstyle\frac\a k}))}
\sqrt{\gcal_\vp({\textstyle\frac\a k})}\pcal_{BF,\delta_k}({\textstyle\frac\a k})(1+R_k({\textstyle\frac\a k}, h_t)),
\end{equation}

\equationspace 
where $R_k=O(k^{\delta-\frac12})$. This expansion is uniform in
$t$ and may be differentiated twice to give for $j=1,2$ and for
some amplitudes $S_j$ of order zero the expansion

\begin{equation}
\label{PDerivativesEq}
\Big(\frac{\pa}{\pa t}\Big)^j
\pcal_{h_t^k}(\a)=C_mk^{\frac12(m-\delta^\sharp_k({\textstyle\frac\a k}))}
\sqrt{\gcal_\vp({\textstyle\frac\a k})}\pcal_{BF,\delta_k}
({\textstyle\frac\a k})(S_j(t,\a,k)+R_k({\textstyle\frac\a k}, h_t)).
\end{equation}
\end{lem}

Next, recall the following asymptotic localization of sums result:

\begin{lem} {\rm(See \cite{SoZ2}, Lemma 1.2.)}
\label{LocalizationLemma} Let $B_k(y,\alpha):kP\cap
\Z^m\rightarrow\C$ be a family of lattice point functions
satisfying $|B_k(y,\alpha)|\le C_0k^M$ for some $C_0,M\ge0$. Fix 
$\delta\in(0,1/2)$. Then there exists $C>0$ such that

$$
\sum_{kP\cap \Z^m} B_k(y,\alpha)\pcal_{h_y^k}(\a,z) = \sum_{\a:
|\frac\a k-\mu_y(z)|\le k^{\delta-\frac12}}
B_k(y,\alpha)\pcal_{h_y^k}(\a,z) +O(k^{-C}).
$$

\end{lem}

\section{\label{LAG} Legendre transform to harmonic functions and Legendre duality of geometric flows}

Over the open orbit of $M$, a toric  \kahler potential may be
identified with a convex function on $\R^m$, $\vp=\vp(\rho)$ in the logarithmic
coordinates of \S \ref{BACKGROUND}. 
By abuse of notation we will frequently identify $\hcal(T)$ with the local
torus-invariant \K potentials defined on the open orbit.
The gradient of $\vp(\rho)=\vp(e^\rho)$ is a one-to-one map, identified with the moment map $\mu$,
whose image is $P$. 
The Legendre transform
appeared in the context of symplectic toric manifolds in the work of Guillemin \cite{Gu} 
and this tool lies at the heart of this section 
(some of our work in this section is also related to work of Semmes \cite{S1}).
It takes
\K potentials $\vp$ on the open orbit to symplectic potentials
$\lcal \vp=u_\vp$, that are defined as convex functions on $P$
with logarithmic singularities on $\partial P$,
and relates the moment map,
local symplectic potential $u_\vp(x)=u_\vp(\mu(\rho))$ and local \K potential
as follows:

\begin{equation}
\label{LegendreTransformEq}
\begin{array}{lll}
u_\vp(x) & =\langle x,2\log\mu^{-1}(x)\rangle-\vp(\mu^{-1}(x))\cr\cr
& = \langle x,(\nabla\vp)^{-1}(x)\rangle-\vp\big((\nabla\vp)^{-1}(x)\big).
\end{array}
\end{equation}

\equationspace
In addition
\begin{equation}
\label{LegendreGradientRelationEq}
(\nabla\vp)^{-1}(x)=\nabla u(x),
\end{equation}
and
\begin{equation}
\label{LegendreHessianRelationEq}
(\nabla^2\vp)^{-1}|_{(\nabla\vp)^{-1}(x)}=\nabla^2 u|_x.
\end{equation}

Any symplectic potential $u$ can be written as $u_0+f$, with respect to the canonical
potential

\begin{equation}
\label{SymplecticPotentialEq}
u_0(x)=\sum_{k=1}^dl_k(x)\log l_k(x)
\end{equation}

\equationspace
introduced by Guillemin, with $f$ smooth up to the boundary \cite{Gu}. 
Just as for \K potentials we may define the space of global symplectic potentials:

\begin{equation}
\lcal\hcal(T)=\{\,f\in C^\infty(P)\,:\, u_0+f=\lcal\vp \hbox{\ \ with\ \ } \vp\in\hcal(T)\}.
\end{equation}

\equationspace
We will sometimes, by abuse of notation, identify elements of this space with their local
symplectic potentials in the same manner as with $\hcal(T)$ itself.

Letting $G_\vp(x)=\nabla^2u_\vp(x)$ we have the following formula of Abreu

\begin{equation}
\label{HessianSymplecticPotentialEq}
\det G_\vp^{-1}=\delta_\vp(x)\cdot\Pi_{r=1}^dl_r(x),
\end{equation}

\equationspace
for some positive smooth function $\delta_\vp$ \cite{A}.

Let $n=\dim_\R N$ and denote by $y_1,\ldots,y_n$ local coordinates over some
coordinate patch $U\subset N$.
We assume that $N$ is oriented as a manifold with boundary, i.e., that the orientation
on $\partial N$ is the one induced from $N$.
Recall that there always exists a Dirichlet Green function
$G(y,q)\in C^\infty(N\times N\setminus\diag(N))$ for the Laplacian 
$\Delta_N:=\Delta_{N,f}$ 
on such  a manifold \cite{Au,GT}. If $v\in C^\infty(\partial 
N)$, the equations
$$
\Delta_N u=0, \hbox{\ on \ } N,
$$
$$
u=v, \hbox{\ on \ } \partial N,
$$
have a unique smooth solution 
\begin{equation}
\label{DirichletHarmonicKernelSolEq}
u(\,\cdot\,)=-\int_{\partial N}v(q)\partial_{\nu(q)} G(\,\cdot\,,q)dV_{\partial N,f}(q)
\end{equation}
(our convention will be that $G(y,q)$ is positive in the interior and vanishes when $q$ is in the boundary),
with $\nu(q)=\nu^i(q)\frac{\pa}{\pa y^i}|_q$ an outward unit normal to $\partial N$ in $N$ (coming from
a Riemannian splitting $TN|_{\partial N}=T\partial N\oplus N_{\partial N}$, where $N_{\partial N}$ is
the normal bundle to $\partial N$ in $N$), and where we let $\partial_{\nu(q)} G(y,q):=\nu(q) G(y,q)\le 0$ be
the normal derivative with respect to the second argument.

Let $\Gamma_{ab}^c$ denote the Christoffel symbols of $(N,f)$ with respect
to local coordinates $y^1,\ldots,y^n$. Recall the following expression for
the Christoffel symbols of $(\hcal,g_{L^2})$. Our proof is a slight variation
on those in \cite{CCh, D1}. 

\begin{lem}
For every $e,f\in T_\vp\hcal$ we have 
$\Gamma(e,f)|_\vp=-{\text\frac12}g_\vp(\nabla e,\nabla f)$.
\end{lem}

\begin{proof}
Recall that the proof of  Koszul formula for the Levi-Civita connection of a finite-dimensional
manifold \cite{P}, page 122,  carries over to infinite-dimensions to show that if a Levi-Civita
connection exists it is unique.
Regard the functions $c,e,f$ as constant vector fields on $\hcal$
and let $D$ denote the Levi-Civita connection of $g_{L^2}$. 
 Therefore the corresponding
brackets vanish and we have

\begin{equation}
\label{KoszulEq}
2g_{L^2}(D_ce,f)|_\vp=c\, g_{L^2}(e,f)-f\,g_{L^2}(e,c)+e\, g_{L^2}(f,c).
\end{equation}
Since
$$
c\, g_{L^2}(e,f)=\frac{d}{dt}\Big|_0\frac1V\int_M ef(\o+t\i\ddbar c)^n=\frac1V\int_M ef\Delta_{\vp}c\,\o^n,
$$
we have
\begin{equation}
\begin{array}{lll}
2g_{L^2}(D_ce,f)|_\vp
& =
\displaystyle\frac1V\int_M (ef\Delta_\vp c-ce\Delta_\vp f + fc\Delta_\vp e)\, \o^n
\cr\cr
& = 
\displaystyle\frac1V\int_M (ef\Delta_\vp c-f\Delta_\vp(ce)+ e\Delta_\vp(fc) )\, \o^n.
\cr\cr
& =
\displaystyle-\frac1V\int_M g_\vp(\nabla c,\nabla e) f\, \o^n.
\end{array}
\end{equation}
It follows that

$$
D_c e = -{\text\frac12} g_{\vp}(\nabla c,\nabla e).
$$

\equationspace
Finally, this expression is symmetric hence $D$ is torsion-free, and it is also compatible with
$g_{L^2}$ since $f\, g_{L^2}(c,e)=g_{L^2}(D_f c,e)+g(c,D_f e)$
is just
$$
\frac1V\int_M ec\Delta_{\o_\vp} f\o^n
=
\frac1V\int_M ec{\text\frac12}\Delta_{g_\vp} f\o^n
=
-\frac1V\int_M{\text\frac12} (e\nabla f\cdot\nabla c+c\nabla f\cdot\nabla e)\o^n.
$$
\end{proof}

Several authors observed previously that
the Legendre transform linearizes the geodesic equation, so
that a geodesic $\vp_t$ with endpoints $\vp_0, \vp_1$ is given by 
$\phi_t =\calL^{-1}( \calL \vp_0 + t (\calL\vp_1 -\calL\vp_0))$ \cite{D1,G,S2}.
We now observe that under the Legendre transform, a harmonic map into $\H(T)$ is
mapped to a family of symplectic potentials that are harmonic
functions in the $N$ variables.
The key point is that the Legendre transform eliminates the Christoffel symbols in a variational sense.

\begin{prop}
\label{LegendreTransformProp}
Let $\psi:\pa N\ra \H(T)$ be a smooth map. There exists a unique harmonic
map $\vp$ from $N$ to $\H(T)$ that agrees with $\psi$ on $\pa N$. Moreover,
$\vp=\lcal^{-1}u$ where $u\in C^\infty(N\times P\setminus\pa P)$
satisfies $\Delta_N u=0$ and $u|_{\partial N}=\lcal\psi$.
\end{prop}

\begin{proof}
The proof of the one-dimensional case (\cite{SoZ2}, Proposition 2.1) 
carries over with minor changes. Indeed, harmonic maps into
$\hcal(T)$ are stationary points of the functional,

\begin{equation}
\label{HarmonicMapEnergyEq}
E(\vp)=\int_{N}|d\vp|^2 dV_{N,f}=\int_{N\times M} f^{ab}\frac{\pa\vp}{\pa
y^a}\frac{\pa\vp}{\pa y^b}\o_\vp^n\wedge dV_{N,f}.
\end{equation}

\equationspace
First considering a variation of (\ref{LegendreTransformEq}) at 
$\rho=\rho(x)=(\nabla \vp)^{-1}(x)$ yields

$$
\frac{\pa u}{\pa y^a}\Big|_x
=
\frac{du}{dy^a}\Big|_x
=
\sum_{j=1}^n x_j\frac{\pa \rho_j}{\pa y^a}
-\frac{\pa \vp}{\pa y^a}\Big|_{\rho}
-\sum_{j=1}^n\frac{\pa \vp}{\pa \rho_j}\frac{\pa \rho_j}{\pa y^a}
=-\frac{\pa \vp}{\pa y^a}\Big|_{\rho},
$$
since $\nabla \vp(\rho)=x$.

Next, torus invariance allows us to integrate instead over the polytope and we have using
(\ref{LegendreGradientRelationEq}) and (\ref{LegendreHessianRelationEq}) that
\begin{equation}
\begin{array}{lll}
(\nabla\vp)_\star (\o_\vp^n) & 
=
(\nabla\vp)_\star \big((\det\nabla^2\vp)d\rho_1\w d\th_1\w\cdots \w d\rho_m\w d\th_m \big)
\cr\cr
& =
(\nabla u)^{-1}_\star \big((\det\nabla^2\vp)d\rho_1\w d\th_1\w\cdots \w d\rho_m\w d\th_m\big)
\cr\cr
& =
(\nabla u)^\star \big((\det\nabla^2\vp)d\rho_1\w d\th_1\w\cdots \w d\rho_m\w d\th_m\big)
\cr\cr
& =
(\det\nabla^2\vp)(\det\nabla^2 u)dx^1\w\cdots\w dx^n=dx.
\end{array}
\end{equation}

\noindent
This means that the metric $g_{L^2}$ is pushed-forward to the `Euclidean' metric on $\calL\H(T)$.
Therefore the functional $E(\vp)$ equals

\begin{equation}
\label{HarmonicEnergyEq}
\int_{N\times P} f^{ab}\frac{\pa u}{\pa y^a}\frac{\pa u}{\pa
y^b}dx\w dV_{N,f} =\int_{N}|d u|^2 dV_{N,f},
\end{equation}
which is equal to the energy of the map $u:N\ra\calL\H(T)$.
Since the target space is now flat with vanishing Christoffel symbols 
the Euler-Lagrange equation is $\Delta_N u=0$.

Finally, note that since $u|_{\pa N}=\lcal\psi|_{\pa N}$ 
is convex on the boundary it is also convex in the interior of $N$: Observe that
from (\ref{DirichletHarmonicKernelSolEq}) it follows that the Hessian of $u$ (in the $P$ variables, namely $x$) 
for every $y\in N$ is given
by
$$
\nabla^2 u(y,x)=-\int_{\partial N}\nabla^2 u(q,x)\partial_{\nu(q)} G(y,q)dV_{\partial N,f}(q),
$$
and since $-\partial_{\nu(q)}G(y,q)\ge0$  (see the paragraph after (\ref{DirichletHarmonicKernelSolEq})) 
it follows that $\nabla^2 u(y,x)$ is therefore a positive-definite matrix. Therefore 
$\vp:=\lcal^{-1}u$ is in $\hcal(T)$ and solves the harmonic map equation with boundary
values $\psi$, as required.
\end{proof}

The Eells-Sampson harmonic map heat flow \cite{ES} on the space of smooth maps from $(N,f)$ to 
$(\hcal(T),g_{L^2})$ is given by

\begin{equation}
\label{EellsSampsonFlowEq}
\pa_t \vp=
f^{ab}\pa_{y^a}\pa_{y^b}\vp-f^{ab}\Gamma_{ab}^c\pa_{y^c}\vp
-{\textstyle\frac12}f^{ab}g(\nabla \pa_{y^a}\vp,\nabla\pa_{y^b}\vp),
\end{equation}

\equationspace
while the heat flow on the space of symplectic potentials $\lcal\hcal(T)$ is 
given by

\begin{equation}
\label{HeatFlowEq}
\pa_t u=\Delta_N u.
\end{equation}

\equationspace
Note that the equations (\ref{EellsSampsonFlowEq}) and (\ref{HeatFlowEq}) hold without
change for the global \K and symplectic potentials, respectively.   

We record the following result
although we will not make use of it for the proof of the main theorem.

\begin{theo}
\label{LegendreFlowThm}
Under the Legendre tranform the Eells-Sampson harmonic map flow (\ref{EellsSampsonFlowEq}) 
on the space of \K potentials $\hcal(T)$ is mapped to the heat flow (\ref{HeatFlowEq}) on 
the space of symplectic potentials $\lcal\hcal(T)$. 
\end{theo}

\begin{proof}
As above, taking a variation of (\ref{LegendreTransformEq}) yields 

\begin{equation}
\label{FirstVariationLegendreEq}
\frac{\pa u}{\pa t}\Big|_x=-\frac{\pa \vp}{\pa t}\Big|_{\rho(x)}.
\end{equation}

\equationspace 
Intuitively, the equality of the energy functionals (\ref{HarmonicMapEnergyEq})
and (\ref{HarmonicEnergyEq}) then suggests that their
Euler-Lagrange equations should coincide, however up to a sign, coming from
the fact that an infinitesimal variation $\delta\psi(\rho)$ in one
corresponds to an infinitesimal variation $-\delta\psi(x)$ in the second. 
More precisely, we have

\begin{equation}
\label{LaplaceLegendreEq}
-\Delta_N u = f^{ab}\pa_{y^a}\pa_{y^b}\vp-f^{ab}\Gamma_{ab}^c\pa_{y^c}\vp
-{\textstyle\frac12}f^{ab}g(\nabla \pa_{y^a}\vp,\nabla\pa_{y^b}\vp).
\end{equation}

\equationspace 
To demonstrate (\ref{LaplaceLegendreEq}), recall first the following formula for the second variation of the 
Legendre duals of a family of convex functions (parametrized by $t$, say) that have the same gradient image, that
follows by taking a variation of (\ref{LegendreTransformEq}) and using (\ref{LegendreGradientRelationEq}):

\begin{equation}
\label{LegendreSecondVariationEq}
\begin{array}{lll}
\dis\frac{\pa^2 u}{\pa t^2}\Big|_x
& =
\dis-\frac{\pa^2 \vp}{\pa t^2}\Big|_{(\nabla\vp)^{-1}(x)}
-\sum_{j=1}^n\frac{\pa^2 \vp}{\pa t \pa \rho_j}\frac{\pa ((\nabla \vp)^{-1}(x))_j}{\pa t}
\cr
& \dis= 
-\frac{\pa^2 \vp}{\pa t^2}\Big|_{(\nabla\vp)^{-1}(x)}
-\sum_{j=1}^n\frac{\pa^2 \vp}{\pa t \pa \rho_j}\frac{\pa (\pa u/\pa x_j)}{\pa t}
\cr
& \dis= 
-\frac{\pa^2 \vp}{\pa t^2}\Big|_{(\nabla\vp)^{-1}(x)}
-\langle \nabla (\pa \vp/\pa t)|_{(\nabla\vp)^{-1}(x)}, \nabla (\pa u/\pa t)|_x\rangle,
\end{array}
\end{equation}

\equationspace 
or more succinctly

\begin{equation}
\label{ClassicalSecondVariationLegendreEq}
-\ddot \vp=\ddot u+\langle\nabla\dot \vp,\nabla\dot u\rangle.
\end{equation}

\equationspace 
Now, the terms that are linear in the first derivatives on each side of (\ref{LaplaceLegendreEq})
are equal to each other by the first variation formula for the Legendre transform (\ref{FirstVariationLegendreEq}). 
Next, fix $y\in N$ and choose coordinates on $N$ for which $f^{ab}=\delta^{ab}$ at $y\in N$.
Then (\ref{ClassicalSecondVariationLegendreEq}) gives

$$
-\pa_{y^a}\pa_{y^a} u=\pa_{y^a}\pa_{y^b}\vp+\langle \nabla\pa_{y^a}u,\nabla\pa_{y^a}\vp\rangle.
$$

\equationspace 
But now 

$$
\pa_{x_j}(\pa_{y^a} u(x))=-\pa_{x_j}(\pa_{y^a}\vp(\rho))
=
-\pa_{\rho_k}\pa_{y^a}\vp\cdot\pa_{x_j}\rho_k
=
-\pa_{\rho_k}\pa_{y^a}\vp\cdot\pa_{x_j}(\nabla u)_k.
$$

\equationspace 
Therefore using (\ref{LegendreHessianRelationEq}) and the fact that $g_\vp$ is represented
in coordinates on the open orbit by $\nabla^2\vp$ we see that (\ref{LaplaceLegendreEq}) holds.
Thus, the Legendre transform sends solutions of (\ref{EellsSampsonFlowEq})
to solutions of (\ref{HeatFlowEq}).
\end{proof}

In general one does not expect the Euler-Lagrange equations of two equal 
functionals defined on two different spaces to transform to each other. In 
our situation this does happen and in essence is due to the fact that the 
Legendre transform eliminates the Christoffel symbols not only in a 
variational sense but pointwise.

Observe that Equation (\ref{LaplaceLegendreEq}) generalizes the well-known formula 
(\ref{ClassicalSecondVariationLegendreEq}) from convex analysis
for the second variation of a family of convex functions on $\R^m$ parametrized by $(\R, dx)$
that have the same gradient image.
The factor $\text\frac12$ in our formula comes from the conventions we used to relate the Riemannian and 
\K metrics.

\section{\label{SectionApproximating} The approximating sequence}

Given a harmonic map $\vp:N \to \hcal(T)$ we now define the purported approximating
sequence of harmonic maps $\vp_k: N \to \bcal_{k}(T)$. 
First, given a
family of toric  \K metrics $\psi$ parametrized by $\partial N$ we
`project' the family pointwise by $\FS_k \circ \Hilb_k$  onto
$\bcal_{k}(T)$ to obtain a family of toric Bergman metrics
parameterized by $\partial N$. Each of these  metrics  is
determined by its  $L^2$ norming constants, hence by the diagonal
matrices

$$
\diag(\qcal_{h^k_{\psi(q)}}(\alpha))_{\alpha\in kP\cap\Z^m}, \quad
q\in \partial N.
$$ 

\equationspace
For each $\alpha$, we solve the boundary problem

$$
\Delta \lambda_{\alpha}(y) = 0, \quad y\in N, 
$$
$$
\lambda_{\alpha}= \log\qcal_{h^k_{\psi(q)}}(\alpha), \quad q\in\partial N.
$$

\equationspace
We then map back to $\bcal_k$ via $\FS_k$ to obtain the family

$$ \vp_k(y,z) =
\frac1k\log\sum_{\alpha\in kP\cap \Z^m} e^{-\lambda_{\alpha}(y)}
|\chi_\alpha(z)|_{h_0^k}^2
 \in\bcal_{k}(T)
$$

\equationspace
of harmonic maps alluded to in Theorem 1.1. 
This may be written somewhat more explicitly in terms of the Green kernel: 

$$
\vp_k(y ,z)
=
\frac1k\log\sum_{\alpha\in kP\cap \Z^m}
|\chi_\alpha(z)|_{h_0^k}^2
\exp\Big(\int_{\partial N} \partial_{\nu(q)} G(y, q)\log||\chi_\alpha||^2_{h_{\psi(q)}}dV_{\pa N,f}(q) \Big).
$$

Our first  aim is to prove the $C^0$ convergence by showing

$$
\mskip-495mu
\vp_k(y,z)-\vp(y,z)=
$$
\begin{equation} \label{C0}
\frac1k\log\!\!\sum_{\alpha\in kP\cap \Z^m}\!
|\chi_\alpha(z)|_{h^k_{\vp(y)}}^2
\exp\Big(\int_{\partial N} \partial_{\nu(q)} G(y, q)\log||\chi_\alpha||^2_{h_{\psi(q)}}dV_{\pa N,f}(q) \Big)
=O({\text\frac{\log k}k}).
\end{equation}
We begin by rewriting the sum in a convenient way.

Put
\begin{equation}
\label{RkEq}
\rcal_k(y,\a):= \exp\Big(-\int_{\partial N} \partial_{\nu(q)} G(y, q)
\log\frac{\qcal_{h^k_{\vp(y)}}(\a)}{\qcal_{h^k_{\psi(q)}}(\a)} dV_{\pa N,f}(q)\Big).
\end{equation}

Then proving (\ref{C0}) is equivalent to proving
\begin{equation}
\label{ConvergenceEq}
\frac1k\log\sum_{\a\in kP\cap\Z^m}\rcal_k(y,\a)\pcal_{h^k_{\vp(y)}}(\a,z)
=O(\log k/k).
\end{equation}

Put $u_y:=u_{\vp(y)}=u(y,\cdot)$, for $y\in N$.
In light of the results in the geodesic case \cite{SoZ2} 
we expect the asymptote of $\rcal_k$ to be the following:

\begin{defin} \label{RINFTY}  
Let the metric volume ratio be the function on $N \times P$ defined by

$$
\rcal_{\infty}(y, x):=
\exp\Big(-{\textstyle\frac12}\int_{\partial N}\partial_{\nu(q)} G(y, q)
\log
\frac {\det \nabla^2 u_y(x)} {\det \nabla^2 u_{q}(x)}dV_{\pa N,f}(q)\Big).
$$
\end{defin}

Note that by (\ref{HessianSymplecticPotentialEq}) we have
$$
\rcal_\infty(y,x)
=
\exp\big(-{\textstyle\frac12}\int_{\partial N}\partial_{\nu(q)} G(y, q)
\log\frac{\delta_{\psi(q)}(x)}{\delta_{\vp(y)}(x)}dV_{\pa N,f}(q)\big),
$$
and therefore $\rcal_\infty\in C^\infty(N\times P)$ (up to the boundary).

In light of Lemma \ref{PLemma} it will be useful to express the ratio $\rcal_k$ in terms
of the functions $\pcal_{h^k}(\a)$ (\ref{PalphaEq}) in the following form:

\begin{lem}
\label{RkLemma}
One has
$$
\rcal_k(y, \alpha)
=
\exp\Big(-\int_{\partial N}\partial_{\nu(q)} G(y, q)
\log\frac{\pcal_{h_{\psi(q)}^k}(\alpha)}{\pcal_{h_{\vp(y)}^k}(\alpha)}dV_{\pa N,f}(q)\Big).
$$

\end{lem}

\begin{proof}
By definition,
$$
\rcal_k(t, \alpha)
=
\exp\Big(\int_{\partial N}\partial_{\nu(q)} G(y, q)
\log\frac{\qcal_{h_{\psi(q)}^k}(\alpha)}{\qcal_{h_{\vp(y)}^k}(\alpha)}dV_{\pa N,f}(q)\Big).
$$

\equationspace
Specializing (\ref{LegendreTransformEq}) to the lattice point $\a$
we have

$$
u_\vp(\al)=\langle\al,2\log\mu^{-1}(\al)\rangle-\vp(\mu^{-1}(\al)),
$$

\equationspace
implying that
\begin{equation}
\label{DualityEq}
\log\qcal_{h^k}(\al)\pcal_{h^k}(\al)=ku(\a/k).
\end{equation}

\equationspace
Since $u$ is harmonic in $y$ it follows that
$$
\log\qcal_{h_{\vp(y)}^k}(\al)\pcal_{h_{\vp(y)}^k}(\al)
=
-\int_{\pa N}\partial_{\nu(q)} G(y, q)\log\qcal_{h_{\psi(q)}^k}(\al)\pcal_{h_{\psi(q)}^k}(\al)dV_{\pa N,f}(q),
$$
which together with the definition concludes the proof.
\end{proof}

\section{\label{SectionProof} Proof of Theorem 1.1}

Recall that the usual elliptic regularity theory applies to the operator 
$\Delta_N$ \cite{Au,GT}. Namely,
there exists $C=C(N,f)$ such that the Schauder estimates hold:

\begin{equation}
\label{SchauderEq}
||u||_{C^{2,1/2}(N)}\le C(||u||_{C^0(N)}+||v||_{C^{2,1/2}(\partial N)}).
\end{equation}

\equationspace
Moreover, the maximum principle implies that $||u||_{C^0(N)}\le ||v||_{C^0(\pa N)}$, and so
the estimates are only in terms of $||v||_{C^{2,1/2}(\partial N)}$.

First, we need the following asymptotic regularity for the coefficients $\rcal_k(t,\a)$.

\begin{lem}
\label{RLemma}
There exists a positive constant $C>0$  such that for all $k,y,\a$ one has
\begin{equation}
\label{RUniformEstimateEq}
1/C<\rcal_k(y,\a)<C.
\end{equation}
Moreover, fixing $\delta\in(0,1/2)$ one has 
\begin{equation}
\label{RkLimitEq}
\rcal_k(y,\a)-\rcal_\infty(y,\a/k)=O(k^{\delta-\frac12})
\end{equation}
uniformly in $(y,\a)$, and 
$\log \rcal_k(y,\a)$ is uniformly bounded in $C^2(N)$.
\end{lem}

\begin{proof}
The Bargmann-Fock terms in (\ref{PEq}) depend only on the geometry of $P$ and not on $y\in N$ and so
are cancelled in the ratio $\rcal_k(y,\a)$. Therefore, by Lemma \ref{RkLemma},

$$\begin{array}{lll}
\log\rcal_k(y,\a)
& = &
\int_{\pa N}-\partial_{\nu(q)} G(y,q)\log\pcal_{h^k_{\psi(q)}}dV_{\pa N,f}(q)
-\log\pcal_{h^k_{\vp(y)}}
\\ && \\
& = &
\frac12\int_{\partial N}-\partial_{\nu(q)} G(y,q)\log\big(\gcal_{\psi(q)}(\frac\a k)\big)dV_{\pa N,f}(q)
-\frac12\log\big(\gcal_{\vp(y)}(\frac\a k)\big)
\\ && \\
&&
+\int_{\partial N}-\partial_{\nu(q)} G(y,q)\log\big(1+R_k(\frac\a k,h_{\psi(q)})\big)dV_{\pa N,f}(q)
\\ && \\ &&
-\log\big(1+R_k(\frac\a k,h_{\vp(y)})\big).
\\ &&
\end{array}
$$
The first two terms simplify to
$$
\frac12
\int_{\partial N}-\partial_{\nu(q)} G(y,q)\log\frac{\delta_{\vp(y)}(\frac\a k)}
{\delta_{\psi(q)}(\frac\a k)}dV_{\pa N,f}(q)
=\log\rcal_\infty(y,\a/ k),
$$
and so the Schauder estimates (\ref{SchauderEq}) together with Lemma \ref{PLemma} 
imply Equation (\ref{RkLimitEq}) and hence also the uniform estimate (\ref{RUniformEstimateEq}).

We now turn to prove the higher derivative estimates.
A first derivative of the fourth term yields, according to (\ref{PDerivativesEq}),
$$
\frac{S_1(y,\a,k)+R_k(\frac\a k,h_{\vp(y)})}{1+R_k(\frac\a k,h_{\vp(y)})},
$$
and this is uniformly bounded according to Lemma \ref{PLemma}. In a similar
fashion it follows that second derivatives are uniformly bounded as well.
Finally, the Schauder estimates (\ref{SchauderEq}) may be invoked again for the third term and these will
be uniform since the same argument as for the fourth term
implies that
$||\log (1+R_k(\frac\a k,h_{\psi(q)}))||_{C^2(\partial N)}$ is uniformly bounded.
\end{proof}

Note that the estimate (\ref{RUniformEstimateEq}) immediately
implies the $C^0$ convergence of $\vp_k$ to $\vp$ with remainder as in
(\ref{C0})
since we have a asymptotic expansion for the \Szego kernel that to first order equals
\begin{equation}
\label{SzegoKernelEq}
\Pi_{h^k_{\vp(y)}}(z,z)=\sum_{\a\in kP\cap\Z^m}\pcal_{h^k_{\vp(y)}}(\a,z)=1+O(k^{-1}).
\end{equation}

We now turn to showing
$C^1$ and $C^2$ convergence. In other words, our aim is now to show that 
the $C^2(N\times M)$ norm of the left hand side of (\ref{ConvergenceEq}) is still $O(k^{\delta-\frac12})$.
In order to prove these estimates it is crucial to make use of some cancellations. These can be understood
as follows. When one replaces all the coefficients $\rcal_k(y,\a)$ by a constant, one reduces to the 
case of a zero-dimensional map, or equivalently to the known asymptotic expansion of the \Szego kernel that
may be differentiated any number of times with a small error.
Now there are two cases. When a coefficient $\rcal_k(y,\a)$ or a derivative thereof only multiplies
a normalized monomial $\pcal_{h^k_{\vp(y)}}(\a,z)$ it is enough to use the uniform estimates
given by Lemma \ref{RLemma} and one does not need to keep track of error terms.
However, as is usually the case, if the coefficient $\rcal_k(y,\a)$ or a derivative thereof multiplies another term
that itself depends on $k$, one needs to keep track of the remainder of order $O(k^{\delta-\frac12})$
given by Lemma \ref{RLemma}. When such an error is introduced we simultaneously apply 
Lemma \ref{LocalizationLemma} to localize to those lattice points satisfying
$|\a|\le k^{\frac12+\delta}$. Remembering the overall factor of $\frac1k$ one then estimates the remainders thus
introduced.

Let us now consider derivatives solely in the $M$-directions.
A derivative of (\ref{ConvergenceEq}) in the $\rho_j$ 
directions amounts to multiplying each coefficient in
the sum (\ref{ConvergenceEq}) by a factor of
$$
k((\nabla\vp_y)(z)-{\textstyle\frac{\alpha_j}{k}})_j=k(\mu_y(z)-{\textstyle\frac{\alpha}{k}})_j.
$$
(recall that the moment map $\mu_y$ is the gradient of the open orbit \K potential 
$\vp(e^\rho)$).

Namely in the interior of $P$ one has,

$$
\frac{\pa}{\pa\rho_j}(\vp_k-\vp)(y,z)=
\frac1k\frac{
\sum_{\a\in kP\cap\Z^m}k(\mu_y(z)-\frac\a k)\rcal_k(y,\a)\pcal_{h_{\vp(y)}^k}(\a)
}
{\sum_{\a\in kP\cap\Z^m}\rcal_k(y,\a)\pcal_{h_{\vp(y)}^k}(\a,z)}.
$$

\equationspace
The factor of $k$ is cancelled by the overall factor of $\frac1k$ and 
the coefficients $\rcal_k(y,\a)$ are uniformly bounded due to (\ref{RUniformEstimateEq}). 
Thus by Lemma \ref{LocalizationLemma} one may restrict to those $\a$ such that 
$|\mu_y(z)-{\textstyle\frac{\alpha}{k}}|\le k^{\delta-\frac12}$ (introducing an error $O(k^{-M})$ for 
some large $M>0$). It follows then that 

$$
\frac{\pa}{\pa\rho_j}(\vp_k-\vp)(y,z)=O(k^{\delta-\frac12}).
$$

\equationspace
Near the boundary of $P$ one performs the same computation but with respect to the slice-orbit coordinates
(the same remark applies to all the computations in this Section).
Note that the argument reduced to the one in \cite{SoZ2}, \S 7.2, once we had (\ref{RUniformEstimateEq}).

Next, we consider second derivatives in the $M$-directions. Symmetrizing sums (see \cite{SoZ2}, \S 8)
one obtains in the interior of $P$,

$$
\mskip-370mu\frac{\pa^2}{\pa\rho_i\pa\rho_j}(\vp_k-\vp)(y,z)=
-\frac{\pa^2 \vp(y,z)}{\pa\rho_i\pa\rho_j}
$$
$$
%\mskip10mu
+\frac1k
\frac
{{\textstyle\frac12}\sum_{\a,\beta\in kP\cap\Z^m}(\a_i-\beta_i)(\a_j-\beta_j)
\rcal_k(y,\a)\rcal_k(y,\be)\pcal_{h^k_{\vp(y)}}(\a,z)\pcal_{h^k_{\vp(y)}}(\beta,z)}
{\sum_{\a\in kP\cap\Z^m}\rcal_k(y,\a)\pcal_{h_{\vp(y)}^k}(\a,z)}.
$$

\equationspace
Equation (\ref{RkLimitEq}) allows to reduce the computations to those in the case $N=[0,1]$:
After localizing (keeping only those $\a$ that are $O(k^{\frac12+\delta})$-close to $k\mu_y(z)$) 
we have the estimate $\frac1k(\a_i-\beta_i)(\a_j-\beta_j)=O(k^{2\delta})$.
We then replace the coefficients $\rcal_k(y,\a)$
by the uniform constant (independent of $\a$) $\rcal_\infty(y,\mu_y(z))$ 
at the price of an error $O(k^{3\delta-\frac12})$. But now what is left is then precisely
cancelled by $-\frac{\pa^2 \vp(y,z)}{\pa\rho_i\pa\rho_j}$ (up to an error of $O(k^{-2})$) due to the 
complete asymptotics of the \Szego kernel of a single metric.
To prove this last claim,
we consider the situation of a family of \Szego kernels parametrized by a compact
manifold $N$, corresponding to the family of Hermitian metrics $h_y,\; y\in N$. 
In the toric situation this may be written explicitly as
$$
\Pi_k(z,z)
:=
\Pi_{h^k_y}(z,z)
=
\sum_{\alpha\in kP\cap \Z^m}
\frac{|\chi_\a(z)|^2_{h_y^k}}{\qcal_{h_y^k}(\a)}
=
\sum_{\alpha\in kP\cap \Z^m}
\frac{e^{\langle\a,\rho\rangle-k\vp_y}}{\qcal_{h_y^k}(\a)}.
$$
Then, $\vp(y,z)+\frac1k\log\Pi_{h^k_y}(z,z)=O(k^{-1})$ (note the minus sign
convention $\o=-\frac\i2\ddbar\vp$) has a complete
asymptotic expansion, and a first space derivative gives
\begin{equation}
\label{SzegoKernelSpaceDerivativeEq}
\frac{\pa\vp(y,z)}{\pa\rho_j}
+
O(k^{-2})
=
\frac{\pa\vp(y,z)}{\pa\rho_j}
+
\frac1k
\frac{\pa\log \Pi_k(z,z)}{\pa\rho_j}
=
\frac1k(\Pi_k(z,z))^{-1}
\sum_{\alpha\in kP\cap \Z^m}
\a_j
\frac{e^{\langle\a,\rho\rangle}}{\calQ_{h_y^k}(\a)}.
\end{equation}

\equationspace
Similarly a second space derivative %is also $O(k^{-1})$ and 
takes the form
\begin{equation}
\label{SzegoKernelSpaceDerivativeEq}
\begin{array}{lll}
\displaystyle\frac{\pa^2\vp(y,z)}{\pa\rho_i\pa\rho_j}+O(k^{-2})
& =\displaystyle
\frac1k(\Pi_k(z,z))^{-2}
\bigg[
\sum_{\alpha\in kP\cap \Z^m}
\a_i\a_j
\frac{\displaystyle e^{\langle\a,\rho\rangle}}{\displaystyle\calQ_{h_y^k}(\a)}
\displaystyle\sum_{\beta\in kP\cap \Z^m}
\frac{\displaystyle e^{\langle\be,\rho\rangle}}{\displaystyle\calQ_{h_y^k}(\be)}
\cr
\displaystyle
&\qquad\qquad\qquad\qquad\displaystyle
-
\!\!
\sum_{\alpha\in kP\cap \Z^m}
\a_i
\frac{\displaystyle e^{\langle\a,\rho\rangle}}{\displaystyle\calQ_{h_y^k}(\a)}
\sum_{\beta\in kP\cap \Z^m}
\be_j
\frac{\displaystyle e^{\langle\be,\rho\rangle}}{\displaystyle\calQ_{h_y^k}(\be)}
\bigg]
\cr
\displaystyle
& =\displaystyle
\frac{(\Pi_k(z,z))^{-2}}k
{\textstyle\frac12}
\sum_{\alpha,\beta\in kP\cap \Z^m}
(\alpha_i-\beta_i)(\alpha_j-\beta_j)
\frac{\displaystyle e^{\langle\a+\be,\rho\rangle}}{\displaystyle\calQ_{h_y^k}(\a)\calQ_{h_y^k}(\be)}
\end{array}
\end{equation}
(by symmetrizing sums).
In conclusion we have proved the claim, and hence the convergence of the second
space derivatives.

Therefore it remains to consider derivatives that also involve the $N$-directions.

We consider first one derivative in the $N$-directions. One has

\begin{equation}
\label{FirstNDerivativeEq}
\begin{array}{lll}
&\dis\frac{\pa}{\pa y^a}(\vp_k-\vp)(y,z)
\cr
&\dis\;=-\frac{\pa\vp}{\pa y^a}
+\frac1k
\frac{
\sum_{\a\in kP\cap\Z^m}\dis
\calR_k(y,\a)\frac{\dis e^{\langle\a,\rho\rangle}}{\dis\calQ_{h^k(y)}(\a)}
\pa_{y^a}\log\frac{\calR_k(y,\a)}{\dis\calQ_{h^k(y)}(\a)}
}
{
\sum_{\a\in kP\cap\Z^m}\dis
\calR_k(y,\a)
\frac{\dis e^{\langle\a,\rho\rangle}}{\dis\calQ_{h^k(y)}(\a)}
}.
\end{array}
\end{equation}

\equationspace
Since the asymptotic expansion (\ref{SzegoKernelEq}) can be differentiated
and is uniform over compact families \cite{C,Z} we have
\begin{equation}
\label{SzegoKernelAsymptotics}
O(k^{-2})
=\dis
\frac1k\frac{\pa}{\pa y^a}\log\Pi_{h^k_{\vp(y)}}(z,z)
=
-\frac{\pa\vp}{\pa y^a}
+\frac1k
\frac{
\sum_{\a\in kP\cap\Z^m}
\frac{\dis e^{\langle\a,\rho\rangle}}{\dis\calQ_{h^k(y)}(\a)}
\pa_{y^a}\log\frac{1}{\dis\calQ_{h^k(y)}(\a)}
}
{\sum_{\a\in kP\cap\Z^m}\frac{\dis e^{\langle\a,\rho\rangle}}
{\dis\calQ_{h^k(y)}(\a)}}
.
\end{equation}
First note that the term
$$
\frac1k
\frac{
\sum_{\a\in kP\cap\Z^m}
\calR_k(y,\a)
\frac{\dis e^{\langle\a,\rho\rangle}}
{\dis\calQ_{h^k(y)}(\a)}
\pa_{y^a}\log\calR_k(y,\a)
}
{\sum_{\a\in kP\cap\Z^m}\frac{\dis e^{\langle\a,\rho\rangle}}
{\dis\calQ_{h^k(y)}(\a)}}
$$
is of order $O(k^{-1})$ by Lemma \ref{RLemma}.
Thus, we are left with the task of comparing the last term of (\ref{SzegoKernelAsymptotics}) with 
$$
\frac1k
\frac{
\sum_{\a\in kP\cap\Z^m}
\calR_k(y,\a)
\frac{\dis e^{\langle\a,\rho\rangle}}
{\calQ_{h^k(y)}(\a)}
\pa_{y^a}\log\frac{1}{\calQ_{h^k(y)}(\a)}
}
{
\sum_{\a\in kP\cap\Z^m}\calR_k(y,\a)\frac{\dis e^{\langle\a,\rho\rangle}}
{\calQ_{h^k(y)}(\a)}
}
$$
We now localize the sums about the image of the moment map using 
Lemma \ref{LocalizationLemma} introducing negligible errors (of arbitarily high order $O(k^{-M})$).
Then we use Lemma \ref{RLemma} to replace each occurrence of $\calR_k(y,\a)$ by
$\calR_\infty(y,\mu_y(z))$ plus an error of order $O(k^{\delta-1/2})$ (the error comes both
from (\ref{RkLimitEq}) and the fact that we Taylor expand $\calR_\infty(y,\mu_y(z))$ about
$\calR_\infty(y,\a/k)$ and use the fact that since we localized the sum we have $|\mu_y(z)-\a/k|=O(k^{\delta-1/2})$).
In the terms involving $\calR_\infty(y,\mu_y(z))$ the factors of $\calR_\infty(y,\mu_y(z))$
actually cancel and so cancel with the last term of (\ref{SzegoKernelAsymptotics}) after localizing the latter.

It remains to show that the error term overall contributes $O(k^{\delta-1/2})$ to the sum
(\ref{FirstNDerivativeEq}). To that end, because of the factor of $1/k$ it is enough to show 
that there exists a uniform constant $C>0$ independent of $k$ such that
\begin{equation}
\label{DerviativeOfQEstimateEq}
|\pa_{y^a}\log\calQ_{h^k(y)}(\a)|\le Ck.
\end{equation}
Recall that the duality (\ref{DualityEq}) of Lemma \ref{RkLemma} implies 
\begin{equation}
\label{QDerivativeEq}
\pa_{y^a}\log\calQ_{h^k_{\vp(y)}}(\a)=k\pa_{y^a} u_y(\a/k)-\pa_{y^a}\log\calP_{h^k_{\vp(y)}}(\a).
\end{equation}
The second term of the right hand side is uniformly bounded by applying (\ref{PDerivativesEq}).
To evaluate the first term recall that
\begin{equation}
\label{uEq}
u_y=-\int_{\partial N}u_q\partial_{\nu(q)} G(y,q)dV_{\partial N,f}(q), \quad y\in N.
\end{equation}
In terms of canonical symplectic potential $u_0$ of (\ref{SymplecticPotentialEq}) one may 
write $u_y=u_0+f_y$ for some globally smooth function $f_y\in\calL\H(T)$ on $P$ and thus we have
\begin{equation}
\label{fEq}
\pa_{y^a} u_y=\pa_{y^a} f_y=-\pa_{y^a}\int_{\partial N}f_q\partial_{\nu(q)} G(y,q)dV_{\partial N,f}(q), \quad y\in N.
\end{equation}
This is uniformly bounded according to the Schauder estimates.
Combining the above the estimate (\ref{DerviativeOfQEstimateEq}) follows.

In sum we have shown that
$$
\frac{\pa}{\pa y^a} (\vp_k-\vp)(y,z)=O(k^{\delta-1/2}),
$$
which concludes the case of a single $N$-derivative.

We now consider the case of mixed second derivatives.
We will always assume $\a,\beta\in kP\cap\Z^n$ and so omit that from the summation notation in what follows. 
To simplify the notation further we will fix a point $(y,z)\in N\times M$ and
use the following abbreviations:

$$
\begin{array}{lll}
& \dis\pa_a:=\pa_{y^a}=\frac{\pa}{\pa y^a},\; \pa_{ab}:=\pa_{y^a}\pa_{y^b}=\frac{\pa^2}{\pa y^a\pa y^b},\cr\cr
& \dis\calR_\a:=\calR_k(y,\a), \; \calQ_\a:=\calQ_{h^k_{\vp(y)}}(\a), \; 
\calP_\a:=\calP_{h^k_{\vp(y)}}(\a), \;
\widetilde\calP_\a:=\frac{\dis e^{\langle\a,\rho\rangle}}
{\calQ_{h^k(y)}(\a)}.
\end{array}
$$

\equationspace
Symmetrizing sums again, it follows that

$$
\mskip-370mu\frac{\pa^2}{\pa y^a\pa\rho^j}(\vp_k-\vp)(y,z)=-\frac{\pa^2\vp(y,z)}{\pa y^a\pa\rho^j}
$$
\begin{equation}
\label{MixedDerivativeFirstEq}
\mskip200mu+\frac1k
\frac{
\frac12
\sum_{\a,\beta}(\a_j-\beta_j)\frac{\pa}{\pa y^a}
\log\left({\displaystyle\frac{\calR_\a}{\calQ_\a}\frac{\calQ_\be}{\calR_\be}}
 \right)
\calR_\a\calR_\beta\widetilde\calP_\a\widetilde\calP_\be
}
{\Big(
\sum_{\a}\calR_\a\widetilde\calP_\a\Big)^2
}
.
\end{equation}
Localizing sums exchanges the term $(\a_j-\beta_j)$ for $O(k^{\frac12+\delta})$
up to an error of $O(k^{-M})$ for some large $M>0$. 
By applying Lemma \ref{RLemma} the term
$$
\frac1k
\frac{
\frac12
\sum_{\a,\beta}(\a_j-\beta_j)\frac{\pa}{\pa y^a}
\log\left({\displaystyle\frac{\calR_\a}{\calR_\be}}
 \right)
\calR_\a\calR_\beta\widetilde\calP_\a\widetilde\calP_\be
}
{\Big(
\sum_{\a}\calR_\a\widetilde\calP_\a\Big)^2
}
$$
is $O(k^{\delta-1/2})$. Now we replace the coefficients
$\calR_k$ in (\ref{MixedDerivativeFirstEq}) for the lattice points $\a$ that remain (near $\mu_y(z)$) 
by the uniform constant $\calR_\infty(y,\mu_y(z))$ plus an error of order
$O(k^{\delta-1/2})$. Any term that does not multiply
$\log\frac{\calQ_\be}{\calQ_\a}$ is of order $O(k^{\delta-1/2})$, or smaller, 
by using Lemma \ref{RLemma}. Now there are two kinds of terms left to estimate. The
first involve $\log\frac{\calQ_\be}{\calQ_\a}$ with all $\calR_k$ replaced by
$\calR_\infty(y,\mu_y(z))$. Then the coefficients $\calR_\infty(y,\mu_y(z))$ cancel
out and we are left with the \Szego kernel approximation of 
$\frac{\pa^2\vp(y,z)}{\pa y^a\pa\rho^j}$ (up to $O(k^{-2})$) and this 
cancels with $-\frac{\pa^2\vp(y,z)}{\pa y^a\pa\rho^j}$ appearing in (\ref{MixedDerivativeFirstEq}).
The second type of contributions comes from error terms of order $O(k^{\delta-1/2})$
multiplying $\log\frac{\calQ_\be}{\calQ_\a}$. By using  (\ref{QDerivativeEq})
we may express $\log\frac{\calQ_\be}{\calQ_\a}$ in terms of the global symplectic potentials 
and using in addition the fact that $\a$ and $\be$ are localized to a neighborhood
of size comparable to $k^{1/2+\delta}$ about $k\mu_y(z)$ we obtain

\begin{equation}
\label{DerivativeQRatioEq}
\Big|\pa_a\log\left(\frac{\calQ_\beta}{\calQ_\a}\right)\Big|
\le 
C+k|\pa_a f_y(\beta/k)-\pa_a f_y(\a/k)|
\le C+kC_1|(\a-\beta)/k|=O(k^{\frac12+\delta}),
\end{equation}

\equationspace
where $C_1$ is the Lipschitz constant of the smooth function $\pa_a f_y$ (as a function of $N$).
By the maximum principle $C_1$ is uniformly bounded in terms of the boundary 
data (i.e., $\FS_k\circ\Hilb_k(\psi)$) since
$\pa_a f_y$ is in fact harmonic in the $N$-variables.
This implies that the second type of contributions are of order $O(k^{2\delta-1})$, due
to the overall factor of $1/k$.

Note that the symmetrization of the sums was crucial here.
In sum,
$$
\frac{\pa^2}{\pa y^a\pa\rho^j}(\vp_k-\vp)(y,z)=O(k^{\delta-\frac12}).
$$

Finally, we consider the case of two derivatives in the $N$-directions.
This case is somewhat more involved than the previous ones and unlike in the case $N=[0,1]$ 
we also need to consider mixed $N$-derivatives.
We have  

\begin{equation}
\label{SecondDerivativeEq}
\begin{array}{lll}
&\mskip-30mu\dis\pa_{ab}(\vp_k-\vp)(y,z)
= -\pa_{ab}\vp
\cr & \dis
\mskip-30mu+ \frac1k
\frac{
\sum_{\a,\beta}\dis\calR_\a\widetilde\calP_\a\calR_\beta\widetilde\calP_\be
\cdot
\Big[
\pa_{ab}
\log
\left(\frac{\calR_\a}{\calQ_\a}\right)+
{\textstyle\frac12}
\pa_a\log
\left(
\frac{\dis\calR_\a}{\dis\calQ_\a}
\frac{\dis\calQ_\beta}{\dis\calR_\beta}
\right)
\pa_b
\log
\left(
\frac{\dis\calR_\a}{\dis\calQ_\a}
\frac{\dis\calQ_\beta}{\dis\calR_\beta}
\right)
\Big]
}
{
\Big(\sum_\a\calR_\a\widetilde\calP_\a\Big)^{-2}
}. 
\end{array}
\end{equation}

\noindent
We rewrite this as 

$$
\pa_{ab}(\vp_k-\vp)(y,z)
= -\pa_{ab}\vp+A+B+C,
$$ 
where 

\begin{equation}
\label{SecondDerivativeSecondEq}
A=
\frac1k
\frac{
\sum_{\a,\beta}\dis\calR_\a\widetilde\calP_\a\calR_\beta\widetilde\calP_\be
\cdot
\Big[
\pa_{ab}
\log
\left(\frac{1}{\dis\calQ_\a}\right)+
{\textstyle\frac12}
\pa_a\log
\left(
\frac{\dis\calQ_\beta}{\dis\calQ_\a}
\right)
\pa_b
\log
\left(
\frac{\dis\calQ_\beta}{\dis\calQ_\a}
\right)
\Big]
}
{
\Big(\sum_\a\dis\calR_\a\widetilde\calP_\a\Big)^{-2}
}, 
\end{equation}

\begin{equation}
\label{SecondDerivativeThirdEq}
B
=\frac1k
\frac{
\sum_{\a,\beta}\dis\calR_\a\widetilde\calP_\a\calR_\beta\widetilde\calP_\be
\cdot
\Big[
\pa_{ab}
\log
\left({\calR_\a}\right)+
{\textstyle\frac12}
\pa_a\log
\left(
\frac{\dis\calR_\a}{\dis\calR_\beta}
\right)
\pa_b
\log
\left(
\frac{\dis\calR_\a}{\dis\calR_\beta}
\right)
\Big]
}
{
\Big(\sum_\a\dis\calR_\a\widetilde\calP_\a\Big)^{-2}
}, 
\end{equation}

\begin{equation}
\label{SecondDerivativeFourthEq}
C
=
\frac1k
\frac{
\sum_{\a,\beta}\dis\calR_\a\widetilde\calP_\a\calR_\beta\widetilde\calP_\be
\cdot
\Big[
\pa_a\log
\frac{\dis\calR_\a}{\dis\calR_\beta}
\pa_b
\log
\frac{\dis\calQ_\beta}{\dis\calQ_\a}
\Big]
}
{
\Big(\sum_\a\dis\calR_\a\widetilde\calP_\a\Big)^{-2}
}. 
\end{equation}

\noindent
By differentiating (\ref{SzegoKernelEq}) we obtain similarly (analogously to the computations leading
to (\ref{SzegoKernelSpaceDerivativeEq})), 
\begin{equation}
\label{SzegoKernelSecondDerivativeEq}
O(k^{-2})=
-\pa_{ab}\vp
+
\frac1k
\frac{
\sum_{\a,\beta}\dis\widetilde\calP_\a\widetilde\calP_\beta
\cdot
\Big[
\pa_{ab}
\log
\left(\frac{1}{\calQ_\a}\right)+
{\textstyle\frac12}
\pa_a\log
\left(
\frac{\dis\calQ_\beta}{\dis\calQ_\a}
\right)
\pa_b
\log
\left(
\frac{\dis\calQ_\beta}{\dis\calQ_\a}
\right)
\Big]
}
{
\Big(\sum_\a\dis\widetilde\calP_\a\Big)^{-2}
}.
\end{equation}

\noindent
We now localize the sums to a ball of radius $k^{\frac12+\delta}$ about $k\mu_y(z)$
introducing a negligible remainder/error and replace the two occurrences of $\calR_\gamma$ 
outside of the square brackets in (\ref{SecondDerivativeEq}) as well as in the denominator 
by the lattice-point-independent constant
$\calR_\infty(y,\mu_y(z))$ (in what follows we refer to this operation as ``replacement") 
introducing an error of order $O(k^{\delta-\frac12})$ for each replacement.
First, observe that the replacement in the denominator is negligible.
What remains to be checked is that the overall error introduced by replacements elsewhere 
is of order $O(k^{\delta-\frac12})$.

In the term $A$ these replacements introduce a term that is cancelled by substituting
the expression for $-\pa_{ab}\vp$ given by (\ref{SzegoKernelSecondDerivativeEq}) 
into (\ref{SecondDerivativeEq}). In addition we introduce error terms. 
Let $\epsilon_\gamma:=\calR_\gamma-\calR_\infty(\mu_y(z))=O(k^{\delta-\frac12})$ 
for each lattice point $\gamma$ in a localized sum.
The highest order remainders are terms of the form
$\frac1k\epsilon_\gamma\cdot (B_1+B_2+B_3)$ where  
$$
\begin{array}{lll}
B_1 & :=B_{1,1}+B_{1,2}:=\dis
\pa_{ab}
\log
\left(\frac{1}{\dis\calQ_\a}\right)+
{\textstyle\frac12}
\pa_a\log
\left(
\frac{\dis\calQ_\beta}{\dis\calQ_\a}
\right)
\pa_b
\log
\left(
\frac{\dis\calQ_\beta}{\dis\calQ_\a}
\right),
\cr
B_2 & :=\dis
\pa_{ab}\log\calR_\a
+
{\textstyle\frac12}
\pa_a\log
\left(
\frac{\dis\calR_\a}{\dis\calR_\beta}
\right)
\pa_b
\log
\left(
\frac{\dis\calR_\a}{\dis\calR_\beta}
\right),
\cr
B_3 & :=
\pa_a\log
\left(	
\frac{\dis\calR_\a}{\dis\calR_\beta}
\right)
\pa_b
\log
\left(
\frac{\dis\calQ_\beta}{\dis\calQ_\a}
\right).
\end{array}
$$
The errors introduced in the replacements in the term $A$ are
of the form $\frac1k\epsilon_\gamma B_1$.
The errors introduced in the replacements of the terms $B$ and $C$
are $\frac1k\epsilon_\gamma B_2$ and  $\frac1k\epsilon_\gamma B_3$, respectively.

First, $\frac1k\epsilon_\gamma B_2=O(k^{\delta-\frac32})$ by Lemma \ref{RLemma}.
To bound $B_1$ and $B_3$ we will use the estimate (\ref{DerivativeQRatioEq}) for 
the derivatives of the norming constants $\calQ_\a$. First, 
it gives directly that $\frac1k\epsilon_\gamma B_3=O(k^{2\delta-1})$.
Second, by squaring (\ref{DerivativeQRatioEq}) we also obtain 
$\frac1k\epsilon_\gamma B_{1,2}=O(k^{3\delta-\frac12})$. 
Finally, the maximum principle also gives, 
similarly to the argument proving (\ref{DerviativeOfQEstimateEq}), that
$$
\Big|\pa_{ab}\log\calQ_\a\Big|\le C+ k|\pa_{ab} f_y(\a/k)|\le C_2k,
$$
and thus $\frac1k\epsilon_\gamma B_{1,1}=O(k^{\delta-\frac12})$.
Altogether we have shown that all the remainders introduced by the replacements are
of order $O(k^{\delta-\frac12})$ for some $\delta\in(0,1/2)$.

Hence we have shown that
$$
\pa_{ab}(\vp_k-\vp)(y,z)=O(k^{\delta-\frac12}),
$$

\equationspace
and this concludes the proof of Theorem 1.1.


\begin{thebibliography}{HHHH}

\bibitem[A]{A} M. Abreu, K\"ahler geometry of toric manifolds in symplectic coordinates,
in: Symplectic and contact topology: interactions and perspectives (Y. Eliashberg et al., Eds.), 
Amer. Math. Soc., 2003, pp. 1--24.

\bibitem[AT]{AT} C. Arezzo, G. Tian, Infinite geodesic rays in the space of K\"ahler potentials,
Ann. Sc. Norm. Super. Pisa Cl. Sci. 2 (2003), 617--630.

\bibitem[Au]{Au} T. Aubin, Some nonlinear problems in Riemannian geometry, Springer, 1998.

\bibitem[B]{B} B. Berndtsson, Positivity of direct image bundles and convexity on the
space of K\"ahler metrics, arxiv:math.CV/0608385.

\bibitem[CCh]{CCh} E. Calabi, X.-X. Chen, The space of K\"ahler metrics II, J. Differential
Geom. 61 (2002), 173--193.

\bibitem[Ca]{Ca} A. Cannas da Silva, Symplectic toric manifolds, 
in: Symplectic geometry of integrable Hamiltonian systems, Birkh\"auser, 2003, pp. 85--173.

\bibitem[C]{C} D. Catlin, The Bergman Kernel and a Theorem of Tian, in: Analysis and geometry in several 
complex variables (G. Komatsu et al., Eds.), Birkh\"auser, 1999, pp. 1--23.

\bibitem[Ch]{Ch} X.-X. Chen, The space of K\"ahler metrics, J. Differential
Geom. 56 (2000), 189--234.

\bibitem[CT]{CT} X.-X. Chen, G. Tian, Geometry of K\"ahler
metrics and foliations by holomorphic discs, arxiv:math.DG/0507148.

\bibitem[D1]{D1} S. K. Donaldson, Symmetric spaces, K\"ahler geometry and Hamiltonian dynamics, in: Northern 
California Symplectic Geometry Seminar (Ya. Eliashberg et al., Eds.), American Mathematical Society, 1999, pp. 13--33.

\bibitem[D2]{D2} S. K.  Donaldson, Scalar curvature and projective
embeddings I, J. Differential Geom. 59 (2001), 479--522.

\bibitem[D3]{D3} S. K. Donaldson, Holomorphic discs and the complex Monge-Amp\`ere
equation, J. Symplectic Geom. 1 (2002), 171--196.

\bibitem[ES]{ES} J. Eells, Jr., J. H. Sampson, Harmonic mappings of Riemannian manifolds,
Amer. J. Math. 86 (1964), 109--160.

\bibitem[GKZ]{GKZ}  I. M. Gelfand, M. M. Kapranov,  A. V. Zelevinsky,
Discriminants, resultants, and multidimensional determinants, Birkh\"auser, Boston, 1994.

\bibitem[GT]{GT} D. Gilbarg, N. S. Trudinger, Elliptic partial differential equations of second order, Springer,
1998.

\bibitem[G]{G} D. Guan, On modified Mabuchi functional and Mabuchi moduli space of K\"ahler metrics on toric bundles,
Math. Res. Lett. 6 (1999), 547--555.

\bibitem[Gu]{Gu} V. W. Guillemin, Kaehler structures on toric varieties, J. Differential Geom. 40 (1994), 285--309.

\bibitem[M]{M} T. Mabuchi, Some symplectic geometry on compact K\"ahler manifolds I, Osaka J.
Math. 24 (1987), no. 2, 227--252.

\bibitem[P]{P} W. A. Poor, Differential geometric structures, McGraw-Hill, 1981.

\bibitem[PS1]{PS1} D. H. Phong, J. Sturm, The Monge-Amp\`ere operator and geodesics in the space of
K\"ahler potentials, Invent. Math. 166 (2006), 125--149.

\bibitem[PS2]{PS2} D. H. Phong, J. Sturm,  Test configurations
for K-stability and geodesic rays, math.DG/0606423.

\bibitem[R]{R} Y. A. Rubinstein, Geometric quantization and dynamical constructions on the space
of K\"ahler metrics, Ph.D. Thesis, Massachusetts Institute of Technology, 2008.

\bibitem[S1]{S1} S. Semmes, Interpolation of Banach spaces, differential geometry, and differential 
equations, Revista Mat. Iberoamericana 4 (1988), 155--176.

\bibitem[S2]{S2} S. Semmes,  Complex Monge-Amp\`ere and symplectic manifolds,
Amer. J. Math. 114 (1992), 495--550.

\bibitem[STZ1]{STZ1} B. Shiffman, T. Tate, S. Zelditch, Distribution laws
for integrable eigenfunctions, Ann. Inst. Fourier (Grenoble) 54 (2004), 1497--1546.

\bibitem[STZ2]{STZ2} B. Shiffman, T. Tate, S. Zelditch, Harmonic analysis on toric varieties, in: Explorations 
in complex and Riemannian geometry (J. Bland et al., Eds.), American Mathematical Society, 2003, pp. 267--286.

\bibitem[SZ]{SZ} B. Shiffman, S. Zelditch,  Asymptotics of almost holomorphic sections of ample 
line bundles on symplectic manifolds, J. Reine Angew. Math.  544 (2002), 181--222.

\bibitem[SoZ1]{SoZ1} J. Song, S. Zelditch, Convergence of Bergman
geodesics on $CP^1$, Annales Inst. Fourier (Grenoble) 57 (2007),
2209--2237.

\bibitem[SoZ2]{SoZ2} J. Song, S. Zelditch, Bergman metrics and geodesics in the space of K\"ahler metrics 
on toric varieties, preprint, arxiv:0707.3082 [math.CV].

\bibitem[SoZ3]{SoZ3} J. Song, S. Zelditch, Test configurations, large deviations and geodesic rays on 
toric varieties, preprint, arxiv:0712.3599 [math.DG].

\bibitem[T]{T} G. Tian, On a set of polarized K\"ahler metrics on algebraic manifolds,
J. Differential Geometry 32 (1990), 99--130.

\bibitem[Z]{Z} S. Zelditch, \szego kernels and a theorem of Tian, Internat. Math. Res. Notices (1998), 317--331.

\end{thebibliography}
\end{document}